\documentclass{article}
\usepackage{fullpage}
\usepackage{amsfonts,amssymb,amsmath,amsthm,dsfont}
\usepackage{graphicx}
\usepackage{comment}
\usepackage[round]{natbib}
\usepackage[colorlinks=true,citecolor=blue]{hyperref}
\usepackage{mdwlist} 

\newtheorem{theorem}{Theorem}
\newtheorem{lemma}[theorem]{Lemma}

\newcommand{\proba}{\mathds{P}}
\newcommand{\bigO}{O}

\newcommand{\reals}{\mathds{R}}
\newcommand{\naturals}{\mathds{Z}_{\geq 0}}
\newcommand{\complex}{\mathds{C}}

\newcommand{\orderings}{{\sf orderings}}
\newcommand{\mgdnm}{\operatorname{MG}^{(\DD)}_{n,m}}
\newcommand{\sgdnm}{\operatorname{SG}^{(\DD)}_{n,m}}
\newcommand{\mgnm}{\operatorname{MG}_{n,m}}
\newcommand{\sgnm}{\operatorname{SG}_{n,m}}

\newcommand{\starmgdnm}{\operatorname{MG}^{(\DD, *)}_{n,m}}
\newcommand{\zeromgdnm}{\operatorname{MG}^{(\DD, 0)}_{n,m}}
\newcommand{\DD}{\mathcal{D}}
\newcommand{\set}{\operatorname{Set}}
\newcommand{\FF}{\mathcal{F}}
\newcommand{\marked}{\operatorname{Marked}}
\newcommand{\weight}{\operatorname{Marked}}

\newcommand{\WW}{\tilde{W}}

\title{Graphs with degree constraints}
\author{\'Elie de Panafieu \footnote{This work was partially founded by the Austrian Science Fund (FWF) grant F5004, the Amadeus program and the PEPS HYDrATA.}\\ \\
Johannes Kepler University \\ \\
depanafieuelie@gmail.com  \and Lander Ramos
\footnote{This work was supported by Spain’s Ministerio de Ciencia e Innovación under the project ``Combinatoria, Teoría de Grafos y Geometría Discreta'' (ref. MTM2011-24097)}\\ \\
Universitat Polit\`ecnica de Catalunya \\ \\
lander.ramos@upc.edu
}

\begin{document}
\maketitle

%

\begin{abstract}
Given a set~$\DD$ of nonnegative integers,
we derive the asymptotic number of graphs
with a given number of vertices, edges,
and such that the degree of every vertex is in $\DD$.
This generalizes existing results,
such as the enumeration of graphs with a given minimum degree,
and establishes new ones,
such as the enumeration of Euler graphs, 
\textit{i.e.} where all vertices have an even degree.
Those results are derived using analytic combinatorics.
\end{abstract}

\section{Introduction}

\subsection{Related works}

The asymptotics of several families of simple graphs 
with degree constraints have been derived.
Regular graphs, where all vertices have the same degree,
have been enumerated by \cite{BC78},
graphs with minimum degree at least~$\delta$ by \cite{PW03}.
An \emph{Euler graph}, or \emph{even graph},
is a graph where all vertices have an even degree.
An exact formula for the number of such graphs,
for a given number of vertices and
without consideration of the number of edges, 
has been derived by \cite{RWR69} and \cite{MS75}.
In the present work, we generalize those results
and derive the asymptotic number of graphs
with degrees in any given set.

A similar problem has been addressed with probabilistic tools
by the \emph{configuration model}, introduced independently
by \cite{B80} and \cite{Wo78}.
This model inputs a distribution~$F$ on the degrees,
and outputs a random multigraph 
where the degree of each vertex follows~$F$.
The main difference with the model analyzed in this article
is that the number of edges in the configuration model
is a random variable.
The link between both models is discussed in Section~\ref{sec:boltzmann}.
For more information on the configuration model,
we recommend the book of \cite{Ho14}.

Other related problems include the enumeration of graphs 
with a given degree sequence (\cite{BC78}),
the enumeration of symmetric matrices 
with nonnegative coefficients and constant row sum (\cite{CMS05}),
and the enumeration of graphs with degree parities,
investigated by~\cite{RR82}.

\subsection{Model and notations}

A \emph{multiset} is an unordered collection of objects,
where repetitions are allowed.
Sets are then multisets without repetitions.
A \emph{sequence} is an ordered multiset.
We use the parenthesis notation~$(u_1, \ldots, u_n)$
for sequences, and the brace notation $\{u_1, \ldots, u_n\}$ for sets and multisets.
Open real intervals are denoted by open square brackets~$]a,b[$.

A \emph{simple graph}~$G$ is a set~$V(G)$ of labelled vertices 
and a set~$E(G)$ of edges, 
where each edge is an unordered pair of distinct vertices.
In a \emph{multigraph}, the edges form a multiset
and the vertices in an edge need not be distinct.
An edge~$\{v,w\}$ is a \emph{loop} if $v = w$,
a \emph{multiple edge} if it has at least two occurrences in the multiset of edges,
and a \emph{simple edge} otherwise.
Thus, the simple graphs are the multigraphs 
that contain neither loops nor multiple edges,
\text{i.e.} that contain only simple edges.
The set of multigraphs with~$n$ vertices and~$m$ edges
is denoted by~$\mgnm$,
and the subset of simple graphs by~$\sgnm$.

The \emph{degree} of a vertex is defined as its number of occurrences in $E(G)$.
In particular, a loop increases its degree by~$2$.
The set of multigraphs from~$\mgnm$ 
where each vertex has its degree in a set~$\DD$
is denoted by $\mgdnm$.
The subset of simple graphs is~$\sgdnm$.
The set $\DD$ may be finite or infinite.
We denote its generating function by
\[
    \set_{\DD}(x) = \sum_{d \in \DD} \frac{x^d}{d!}.
\]
For any natural number~$i$, $\DD-i$ denotes the set
$\{d-i \in \naturals\ |\ d \in \DD\}$.
In particular, observe that $\set_{\DD}'(x) = \set_{\DD-1}(x)$.
We also define the \emph{valuation} $r = \min(\DD)$ 
and \emph{periodicity} $p = \gcd \{d_1-d_2\ |\ d_1,d_2 \in \DD\}$ of the set~$\DD$
(by convention, the periodicity is infinite when~$|\DD| = 1$).

\section{Main Theorem and applications}

Our main result is an asymptotic expression
for the number of graphs in~$\sgdnm$,
when the number~$m$ of edges grows linearly 
with the number~$n$ of vertices.

\begin{theorem} \label{th:simple}
Assume~$\DD$ contains at least two integers,
has valuation~$r = \min \{d \in \DD\}$ 
and periodicity~$p = \gcd \{ d_1 - d_2\ |\ d_1, d_2 \in \DD\}$.
Let~$m$, $n$ denote two integers tending to infinity,
such that $2m/n$ stays in a fixed compact interval of $]r,\max(\DD)[$
and~$p$ divides $2m - r n$, then
the number of simple graphs in $\sgdnm$ is
\[
	\frac{(2m)!}{2^m m!}
    \frac{\set_{\DD}(\zeta)^n}{\zeta^{2m}}
    \frac{p}{\sqrt{2 \pi n \zeta \phi'(\zeta)}}
    e^{- W_{\frac{n}{m}}(\zeta)^2 - W_{\frac{n}{m}}(\zeta)}
    (1 + \bigO(n^{-1})),
\]
where $\phi(x) = \frac{x \set_{\DD-1}(x)}{\set_{\DD}(x)}$,
$\zeta$ is the unique positive solution of
$\phi(\zeta) = \frac{2m}{n}$,
and $W_{\frac{n}{m}}(x) = \frac{n}{4m} \frac{x^2 \set_{\DD-2}(x)}{\set_{\DD}(x)}$.
If~$p$ does not divide $2m - r n$, if $2m/n < r$ or if $2m/n > \max(\DD)$,
then $\sgdnm$ is empty.
\end{theorem}

When $\DD = \naturals$, the degrees are not constrained, so $\sgdnm = \sgnm$.
Using Stirling formula, it can indeed be checked that~$\binom{\binom{n}{2}}{m}$, 
the total number of simple graphs with~$n$ vertices and~$m$ edges,
is asymptotically equal to the result of Theorem~\ref{th:simple}
\[    
    \frac{n^{2m}}{2^m m!}
    \frac{(2m)!}{(2m)^{2m} e^{-2m} \sqrt{2 \pi 2m}}
    e^{-\left(\frac{m}{n}\right)^2 - \frac{m}{n}}
    \left( 1 + \bigO(n^{-1}) \right).
\]

\cite{PW03} have derived the asymptotics
of simple graphs with minimum degree at least~$\delta$.
They used probabilitic and analytic elementary tools,
in a sophisticated way.
In the present paper, we have addressed the enumeration
of a broader family of graphs with degree constraints,
using more powerful tools (analytic combinatorics).
For graphs with minimum degree at least~$\delta$,
the asymptotics derived in Theorem~\ref{th:simple}, 
for $\DD = \mathds{Z}_{\geq \delta}$,
matches their result.

Euler graphs are simple graphs where each vertex has an even degree.
An exact, but complicated, formula for the number of such graphs,
for given number of vertices and
without consideration of the number of edges, 
has been derived by \cite{RWR69} and \cite{MS75}.
Applying Theorem~\ref{th:simple}, we are now able to derive
the asymptotic number of Euler graphs with~$n$ vertices and~$m$ edges,
when~$2m/n$ stays in a fixed compact interval of~$\reals_{>0}$
\[
    |\operatorname{SG}^{(\text{even})}_{n,m}|
    =
    \frac{(2m)!}{2^m m!}
    \frac{\cosh(\zeta)^n}{\zeta^{2m}}
    \frac{2}{\sqrt{2 \pi n \zeta \phi'(\zeta)}}
    e^{- \left( \frac{n}{4m} \zeta^2 \right)^2 - \frac{n}{4m} \zeta^2}
    (1 + \bigO(n^{-1})),
\]
where~$\phi(x) = x \tanh(x)$ and~$\tanh(\zeta) = 2m/n$.

\section{Proof of the result}

In this section we provide a proof for
Theorem~\ref{th:simple}. The proof of all lemmas and theorems
are moved to the appendix.

    \subsection{Preliminaries} \label{sec:model}

\subsubsection{Multigraph model}

The main model of random multigraphs with~$n$ vertices and~$m$ edges
is the \emph{multigraph process}, analyzed by \cite{FKP89} and \cite{JKLP93}.
It samples uniformly and independently~$2m$ vertices 
$(v_1, v_2, \ldots, v_{2m})$ in $\{1, \ldots, n\}$,
and outputs a multigraph with set of vertices $\{1, \ldots, n\}$
and set of edges
\[
	\{ \{v_{2i-1}, v_{2i}\}\ |\ 1 \leq i \leq m\}.
\]

Given a simple or multi graph, one can order the set of edges and the vertices in each edge.
The result is a sequence of ordered pairs of vertices, that we call an \emph{ordering} of~$G$.
Let $\orderings(G)$ denote the number of such orderings.
For example, the multigraph on~$2$ vertices with edges~$\{\{1,1\},\{1,2\},\{1,2\}\}$
has~$12$ orderings, amongst them~$((1,2),(1,1),(2,1))$.
For simple graphs, the number of orderings is equal to~$2^m m!$, 
because each edge has two possible orientations
and all edges can be permuted.
For non-simple multigraphs, $\orderings$ is smaller.
\cite{FKP89} and \cite{JKLP93} introduced the \emph{compensation factor}~$\kappa(G)$ 
of a multigraph~$G$ with $m$ edges, defined as
\[
    \kappa(G) = \frac{\orderings(G)}{2^m m!}.
\]
The compensation factor of a multigraph is~$1$ 
if and only if it is simple.

Observe that in the random distribution
induced by the multigraph process, 
each multigraph receives a probability
proportional to its compensation factor.
Therefore, when the output of the multigraph process
is constrained to be a simple graph,
the sampling becomes uniform on~$\sgnm$.
The \emph{total weight} of a family~$\FF$ of multigraphs is the sum
of their compensation factors.
For example, the total weight of~$\mgnm$ is equal to~$\frac{n^{2m}}{2^m m!}$.
When~$\FF$ contains only simple graphs, its total weight
is equal to its cardinality.

    \subsubsection{Analytic tools} \label{sec:notations}

Our tool for the analysis of graphs with degree constraints
is \emph{analytic combinatorics}, as presented by \cite{FS09}.
Its principle is to associate to the combinatorial family studied
its \emph{generating function}. The asymptotics of the family
is then linked to the analytic behavior of this function.

In the analysis of a graphs family~$\mathcal{F}$
with analytic combinatorics,
the main difficulty is the fast growth of its cardinality,
which often implies a zero radius of convergence
for the corresponding generating function
\[
    \sum_{G \in \mathcal{F}}
    w^{|E(G)|}
    \frac{z^{|V(G)|}}{|V(G)|!}.
\]
This feature drastically reduces the number of tools 
from complex analysis that can be applied.
%
Graphs with degree constraints are no exception,
but our approach completely avoid this classic issue.
In fact, the only analytic tool we use
is the following lemma,
a variant of \cite[Theorem VIII.8]{FS09}.

\begin{lemma} \label{th:large_power}
Consider a non-monomial series~$B(z)$ with nonnegative coefficients, 
analytic on~$\complex$,
with valuation $r = \min\{n\ |\ [z^n] B(z) \neq 0\}$
and periodicity $p = \gcd\{n\ |\ [z^{n-r}] B(z) \neq 0\}$.
Let~$\phi(z)$ denote the function $\frac{z B'(z)}{B(z)}$,
and $K$ a compact interval of the open interval $]r, \lim_{x \to \infty} \phi(x)[$.
Let~$N$, $n$ denote two integers tending to infinity
while~$N/n$ stays in~$K$,
and let~$\zeta$ denote the unique positive solution of~$\phi(\zeta) = N/n$.
Finally, consider a compact~$Y$ and a function~$A(y,z)$, $\mathcal{C}^2$ on $Y \times \complex$,
such that for all~$y$ in~$Y$, the function~$z \mapsto A(y,z)$ is analytic on~$\complex$
and~$A(y, \zeta^p)$ is nonzero.
Then we have, uniformly for~$N/n$ in~$K$ and~$y$ in~$Y$,
\[
    [z^N]
    A(y, z^p)
    B(z)^n
    =
    \begin{cases}
    \frac{p A(y, \zeta^p)}{\sqrt{2 \pi n \zeta \phi'(\zeta)}}
    \frac{B(\zeta)^n}{\zeta^{N}}
    \left(1 + \bigO(n^{-1}) \right)
    & \text{ if $p$ divides $N - n r$},
    \\
    0
    & \text{ otherwise}.
    \end{cases}
\]
\end{lemma}

\subsection{Multigraphs with degree constraints}

The work of \cite{FKP89} and \cite{JKLP93}
demonstrates that multigraphs are more suitable
to the analytic combinatorics approach than simple graphs.
Moreover, the results on multigraphs can usually
be extended to simple graphs.
Following this observation, multigraphs are analyzed in this section,
before turning so simple graphs in Section~\ref{sec:simple}.

    \subsubsection{Exact and asymptotic enumeration}

We derive an exact expression for the number of multigraphs with degree constraints
in Theorem~\ref{th:multigraphs},
then translates it into an asymptotics
in Theorem~\ref{th:multigraphs_asymptotics}.

\begin{theorem} \label{th:multigraphs}
The total weight of all multigraphs in $\mgdnm$ is
\[
    \sum_{G \in \mgdnm} \kappa(G) 
    =
    \frac{(2m)!}{2^m m!} 
    [x^{2m}] 
    \set_{\DD}(x)^n.
\]
\end{theorem}

The proof of this theorem is elementary by the definition
of the compensation factor.
Now applying Lemma~\ref{th:large_power} to the exact expression,
we derive the asymptotics of multigraphs with degree constraints.
Let us first eliminate three simple cases.
\begin{itemize}
\item
When~$\DD$ contains only one integer $\DD = \{d\}$,
$\mgdnm$ is the set of $d$-regular multigraphs.
The total weight of $\mgdnm$ is then~$0$
if $2m \neq n d$, and
$
    \frac{(2m)!}{2^m m! d!^n}
$ 
otherwise.
\item
The sum of the degrees of the vertices
is equal to~$2 m$, so~$\mgdnm$ is empty
when $2m/n < \min(\DD)$ or $2m/n > \max(\DD)$.
\item
The periodicity~$p$ of~$\DD$ is equal to~$\gcd\{ d - r\ |\ d \in \DD\}$.
For each vertex~$v$ of a multigraph from~$\mgdnm$,
it follows that~$p$ divides $\deg(v) - r$.
By summation over all vertices, we conclude that if~$p$ does not divide $2m - n r$,
then the set~$\mgdnm$ is empty.
\end{itemize}
The two last points obviously hold for~$\sgdnm$.

\begin{theorem} \label{th:multigraphs_asymptotics}
Consider a set~$\DD \subset \naturals$ of size at least~$2$.
Let $r = \min(\DD)$ denote its valuation
and $p = \gcd \{d_1 - d_2\ |\ d_1, d_2 \in \DD\}$ its periodicity.
Let~$m$, $n$ denote two integers tending to infinity,
such that $2m/n$ stays in a fixed compact interval of the open interval $]r, \max(\DD)[$,
and~$p$ divides~$2m - r n$, 
then the total weight of~$\mgdnm$ is equal to
\[
    \sum_{G \in \mgdnm} \kappa(G)
    =
    \frac{(2m)!}{2^m m!}
    \frac{p}{\sqrt{2 \pi n \zeta \phi'(\zeta)}}
    \frac{\set_{\DD}(\zeta)^n}{\zeta^{2m}}
    \left(1 + \bigO(n^{-1})\right)
\]
where $\phi(x) = \frac{x \set_{\DD-1}(x)}{\set_{\DD}(x)}$
and~$\zeta$ is the unique positive solution of $\phi(\zeta) = \frac{2m}{n}$.
\end{theorem}

    \subsubsection{Typical multigraphs with degree constraints} \label{sec:typical}

Let us recall that an edge is \emph{simple}
if it is neither a loop nor a multiple edge.
Before turning to the enumeration of simple graphs with degree constraints,
we first describe the behavior of non-simple edges
in a typical multigraph from~$\mgdnm$.
No proofs are given here, as stronger results
will be derived later.

Using random sampling, we observe 
that in most of the multigraphs from~$\mgdnm$,
all non-simple edges have low multiplicity and are well separated.
This motivates the following definition.
A multigraph from~$\mgdnm$ is in~$\starmgdnm$
if all its non-simple edges are loops or double edges,
and each vertex belongs to at most one loop or (exclusive) one double edge.
Let~$|E|_e$ denote the number of occurrences of the element~$e$ in the multiset~$E$.
Formally, $\starmgdnm$ is characterized
as the set of multigraphs~$G$ from~$\mgdnm$ 
such that for all vertices~$u,v,w$, we have
\[
\begin{array}{lcl}
    |E(G)|_{\{v,v\}} \leq 1,
    && |E(G)|_{\{u,v\}} = |E(G)|_{\{v,w\}} = 2 \implies u=w,\\[0.1cm]
    |E(G)|_{\{v,w\}} \leq 2,
    && \{v,v\} \in E(G) \implies \forall w,\ |E(G)|_{\{v,w\}} \leq 1.
\end{array}
\]
The complementary set, $\mgdnm \setminus \starmgdnm$,
is denoted by~$\zeromgdnm$,
and illustrated in Figure~\ref{fig:nonstar}.
\begin{figure}
\centerline{\includegraphics{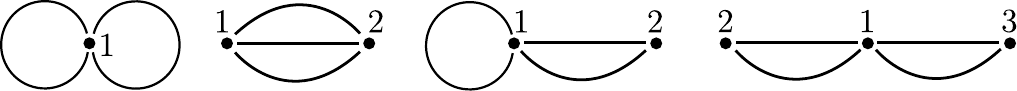}}
\caption{Four examples of multigraphs from~$\zeromgdnm$.}
\label{fig:nonstar}
\end{figure}

\subsection{Simple graphs with degree constraints} \label{sec:simple}

We introduce the notation~$\sgdnm$ for the set
of simple graphs with~$n$ vertices, $m$ edges
and all degrees in~$\DD$,
\textit{i.e.} multigraphs from~$\mgdnm$ that contain
neither loops nor multiple edges.
The enumeration of simple graphs with degree constraints
is derived in Theorem~\ref{th:simple}.
First, in Section~\ref{sec:inclusion_exclusion_star},
we describe an inclusion-exclusion process that outputs~$|\sgdnm|$
when applied to~$\starmgdnm$.
In Section~\ref{sec:inclusion_exclusion_all},
this process is then applied to~$\mgdnm$,
and the error introduced is proven to be negligible
in Section~\ref{sec:negligible}.

In order to forbid loops and multiple edges in multigraphs from~$\mgdnm$,
we introduce the notion of \emph{marked multigraphs}.

    \subsubsection{Marked multigraphs}

A \emph{marked multigraph}~$G$ is
a triplet $(V(G), E(G), \bar{E}(G))$,
where~$V(G)$ denotes the set of vertices,
$E(G)$ the multiset of \emph{normal edges},
and~$\bar{E}(G)$ the multiset of \emph{marked edges},
where both normal and marked edges are unordered pairs of vertices.
We say that a marked multigraph~$G$ belongs to a family~$\FF$ of (unmarked) multigraphs
if the unmarked multigraph $(V(G), E(G) \cup \bar{E}(G))$ is in~$\FF$.

We now extend to marked multigraphs the definitions
of degree, orderings and compensation factors, 
introduced for multigraphs in Section~\ref{sec:model}.
The \emph{degree} of a vertex from a marked multigraph~$G$
is equal to its number of occurrences in the multiset~$E(G) \cup \bar{E}(G)$.
An \emph{ordering} of a marked multigraph~$G$ 
with $m = |E(G)| + |\bar{E}(G)|$ edges is a sequence
\[
    S = ((v_1,w_1,t_1), \ldots, (v_m, w_m, t_m))
\]
from~$(V(G) \times V(G) \times \{0,1\})^m$ such that
the multiset
$
    \{\{v_i,w_i\}\ |\ (v_i, w_i, 0) \in S\}
$
is equal to~$E(G)$, and the multiset
$
    \{\{v_i,w_i\}\ |\ (v_i, w_i, 1) \in S\}
$
is equal to~$\bar{E}(G)$.
The number of orderings of a given marked multigraph~$G$
is denoted by~$\orderings(G)$, and its \emph{compensation factor} is
\[
    \kappa(G) = \frac{\orderings(G)}{2^m m!}.
\]
For example, consider the marked multigraph~$G$ with
\[
    V(G) = \{1,2\}, \quad E(G) = \{\{1,2\}\}, \quad \bar{E}(G) = \{\{1,2\},\{1,2\}\}.
\]
Its number of orderings is~$24$, and therefore its compensation factor is
$
    \kappa(G) = 1/2,
$
whereas it is~$1/6$ for~$G$ without the marks,
\[
    V(G) = \{1,2\}, \quad E(G) = \{\{1,2\},\{1,2\}, \{1,2\}\}.
\]

In the following, we will consider families of marked multigraphs
where the marked edges are loops or multiple edges.
Given a marked multigraph~$G$, then~$\ell(G)$ denotes
the number of loops in~$\bar{E}(G)$,
and~$k(G)$ the number of distinct edges from~$\bar{E}(G)$
that are not loops.
The generating function of a family~$\FF$ or marked multigraphs is
\[
    F(u,v) 
    = 
    \sum_{G \in \FF}
    \kappa(G)
    u^{k(G)}
    v^{\ell(G)}.
\]

    \subsubsection{Inclusion-exclusion process} \label{sec:inclusion_exclusion_star}

In this section, we build an operator~$\marked$
that inputs a family of multigraphs
and outputs a family of marked multigraphs.
It is designed so that the asymptotics of its generating function~$\marked_{\mgdnm}(u,v)$
is linked to the asymptotics of~$|\sgdnm|$.
In order to justify the construction, 
we first introduce the operators~$\marked^{(1)}$
and~$\marked^{(2)}$.

\paragraph{First marking.}
If we could mark all loops and multiple edges from~$\mgdnm$,
the enumeration of simple graphs with degree constraints would be easy.
Indeed, given a family~$\FF$ of multigraphs, let~$\marked_{\FF}^{(1)}$
denote the marked multigraphs from~$\FF$ with all loops
and multiple edges marked.
Since the simple graphs are the multigraphs that have neither loops nor multiple edges,
we have
\[
    \marked_{\mgdnm}^{(1)}(0,0) 
    = 
    \sum_{G \in \mgdnm}
    \kappa(G) 0^{k(G)} 0^{\ell(G)}
    =
    \sum_{G \in \sgdnm}
    \kappa(G),
\]
which is equal to~$|\sgdnm|$, because simple graphs
have a compensation factor equal to~$1$.
Unfortunately, we do not have a description of this family
in the symbolic method formalism.

\paragraph{Second marking.}
The inclusion-exclusion principle advises us to mark \emph{some} of the non-simple edges.
Let~$\marked_{\FF}^{(2)}$ denote the set of marked multigraphs~$G$ from~$\FF$
such that each edge from~$\bar{E}(G)$ is either a loop,
or has multiplicity at least~$2$ in~$\bar{E}(G)$
and does not belong to~$E(G)$.
This construction implies the relation
\[
    \marked_{\FF}^{(2)}(u,v)
    =
    \marked_{\FF}^{(1)}(u+1, v+1),
\]
and therefore
\[
    |\sgdnm| = \marked_{\mgdnm}^{(2)}(-1,-1).
\]
The natural idea to build a marked multigraph~$G$ from~$\marked_{\mgdnm}^{(2)}$
is to first choose some loops and multiple edges to put in~$\bar{E}(G)$,
then complete~$E(G)$ with unmarked edges,
which may well form other loops and multiple edges,
in a way that ensures~$G \in \marked_{\mgdnm}^{(2)}$.
However, the description of the set of marked edges is complicated, 
because of the numerous possible intersection patterns.

\paragraph{Third marking.}
We have seen in Section~\ref{sec:typical}
that in most of the multigraphs from~$\mgdnm$,
non-simple edges do not intersect.
This motivates the following definition.
Given a set~$\FF$ of multigraphs, let $\marked(\FF)$ denote 
the set of marked multigraphs from~$\FF$ such that 
each vertex is in exactly one of the following cases:
\begin{itemize}
\item
the vertex belongs to no marked edge,
\item
the vertex belongs to one marked loop and no other marked edge,
\item
the vertex belongs to two identical marked edges and no other marked edge.
\end{itemize}
Therefore, each marked edge is a loop of multiplicity~$1$ or a double edge.
This marking process links the multigraphs from~$\starmgdnm$,
defined in Section~\ref{sec:typical}, to the simple graphs with degree constraints.

\begin{lemma} \label{th:inclusion_exclusion}
The value $\marked_{\starmgdnm}(-1,-1)$
is equal to the number of simple graphs in~$\sgdnm$.
\end{lemma}

Applying the operator~$\marked$ to the decomposition
\[
    \mgdnm
    =
    \starmgdnm
    \uplus
    \zeromgdnm,
\]
we find
\[
    \marked_{\mgdnm}(u,v)
    =
    \marked_{\starmgdnm}(u,v)
    +
    \marked_{\zeromgdnm}(u,v)
\]
which implies, after evaluation at~$(u,v) = (-1,-1)$ and reordering of the terms, 
\[
    |\sgdnm|
    =
    \marked_{\mgdnm}(-1,-1)
    -
    \marked_{\zeromgdnm}(-1,-1).
\]
We compute the asymptotics of $\marked_{\mgdnm}(-1,-1)$
in Section~\ref{sec:inclusion_exclusion_all},
and prove that $\marked_{\zeromgdnm}(-1,-1)$
is negligible in Section~\ref{sec:negligible}.

    \subsubsection{Application of the inclusion-exclusion process to all multigraphs with degree constraints} \label{sec:inclusion_exclusion_all}

We start with an exact expression 
of~$\weight_{\mgdnm}(u,v)$ in Lemma~\ref{th:exact_marked},
then derive its asymptotics in Lemma~\ref{th:asympt_marked}.

\begin{lemma} \label{th:exact_marked}
We have the formal equality
\[
	\weight_{\mgdnm}(u,v)
    =
    \frac{(2m)!}{2^m m!}
    [x^{2m}]
    \Bigg(
    \sum_{k,\ell \geq 0}
    a_{n,m,2k+\ell}
    \frac{(u W_{\frac{n}{m}}(x)^2)^k}{k!}
    \frac{(v W_{\frac{n}{m}}(x))^{\ell}}{\ell!}
    \Bigg)
    \set_{\DD}(x)^n,
\]
where $a_{n,m,j} = 0$ when $j$ is greater than $\min(n,m)$, otherwise
\begin{align*}
    a_{n,m,j}
    &=
    \frac{n!}{(n-j)! n^{j}}
    \frac{m!}{(m-j)! m^{j}}
    \frac{(2m-2j)! (2m)^{2j}}{(2m)!},
    \\
    W_{\frac{n}{m}}(x)
    &=
    \frac{n}{4m}
    \frac{x^2 \set_{\DD-2}(x)}{\set_{\DD}(x)}.
\end{align*}
\end{lemma}

The proof is constructive by considering all the disjoint sets of vertices
where we can put a loop or a double edges.
We observe that 
when~$2k+\ell$ is fixed while~$n,m$ tends to infinity,
then~$a_{n,m,2k+\ell}$ tends to~$1$.
The double sum can then be approximated by an exponential,
and it is tempting to conclude
\[
    \weight_{\mgdnm}(u,v)
    \sim
    [x^{2m}]
    e^{u W_{\frac{n}{m}}(x)^2 + v W_{\frac{n}{m}}(x)}
    \set_{\DD}(x)^n.
\]
The next lemma formalize this intuition.
A multivariate generating function $f(x_1, \ldots, x_n)$ is said to
\emph{dominate coefficient-wise} another series $g(x_1, \ldots, x_n)$ if
for all $k_1, \ldots, k_n \geq 0$,
\[
	\left|
	[x_1^{k_1} \cdots x_n^{k_n}]
	g(x_1, \ldots, x_n)
	\right|
	\leq
	[x_1^{k_1} \cdots x_n^{k_n}]
	f(x_1, \ldots, x_n).
\]

\begin{lemma} \label{th:technical}
When~$m/n$ stays in a fixed compact interval of~$\reals_{>0}$,
there is an entire bivariate analytic function~$C(u,v)$ such that,
for~$n$ large enough, 
$\frac{1}{n} C(u,v)$ dominates coefficient-wise
\begin{equation} \label{eq:technical}
	e^{u+v}
	-
	\sum_{k,\ell \geq 0}
	a_{n,m,2k+\ell}
    \frac{u^{k}}{k!}
    \frac{v^{\ell}}{\ell!}.
\end{equation}
\end{lemma}

We can now derive the asymptotics of~$\marked_{\mgdnm}(u,v)$.
As observed in the discussion preceding Theorem~\ref{th:multigraphs_asymptotics},
the result is trivial when $\DD$ contains only one integer,
when $2m/n$ is outside~$[\min(\DD), \max(\DD)]$
and when $p$ does not divide $2m - \min(\DD) n$.

\begin{lemma} \label{th:asympt_marked}
Assume $\DD$ has size at least~$2$, 
valuation~$r$ and periodicity~$p$.
Let~$m$, $n$ denote two integers tending to infinity,
such that $2m/n$ stays in a fixed compact interval of $]r, \max(\DD)[$
and~$p$ divides $2m - r n$.
When $u$, $v$ stay in a fixed compact, then
\begin{equation} \label{eq:asympt_marked}
    \weight_{\mgdnm}(u,v)
    =
    \frac{(2m)!}{2^m m!}
    \frac{\set_{\DD}(\zeta)^n}{\zeta^{2m}}
    \frac{p}{\sqrt{2 \pi n \zeta \phi'(\zeta)}}
    e^{u W_{\frac{n}{m}}(\zeta)^2 + v W_{\frac{n}{m}}(\zeta)}
    (1 + \bigO(n^{-1})),
\end{equation}
where $W_{\frac{n}{m}}(x) = \frac{n}{4m} \frac{x^2 \set_{\DD-2}(x)}{\set_{\DD}(x)}$, 
$\phi(x) = \frac{x \set_{\DD-1}(x)}{\set_{\DD}(x)}$
and $\phi(\zeta) = \frac{2m}{n}$.
\end{lemma}

The proof is a consequence of Lemma~\ref{th:large_power}, Lemma~\ref{th:exact_marked}
and Lemma~\ref{th:technical}. 

    \subsubsection{Negligible marked multigraphs} \label{sec:negligible}

Recall that~$\zeromgdnm$ denotes the set $\mgdnm \setminus \starmgdnm$.
In Lemma~\ref{th:large_excess}, we prove that $\marked_{\zeromgdnm}(-1,-1)$ is negligible.
To do so, we first bound $\marked_{R}(1,1)$
for a family~$R$ of marked multigraphs from~$\mgdnm$
with mandatory edges.

\begin{lemma} \label{th:mandatory_edges}
Let $e_1, \ldots, e_j$ 
denote~$j$ edges on the set of vertices $\{1,\ldots,n\}$, 
and~$R$ the set of multigraphs from $\mgdnm$ 
that contain those edges, with multiplicities
(\textit{i.e.} an edge with~$k$ occurrences in the list 
has at least~$k$ occurrences in the multiset of edges of the multigraph)
\[
	R
	=
	\Big\{ 
	G \in \mgdnm\ \Big|\ 
    \forall 1 \leq i \leq j,\ 
	e_i \in E(G) 
    \text{ with multiplicities}
	\Big\}.
\]
Assume~$\DD$ contains at least two integers and has valuation~$r$.
Let~$m$, $n$ denote two integers tending to infinity,
such that $2m/n$ stays in a fixed compact interval of $]r, \max(\DD)[$, then
\[
	\weight_{R}(1,1)
	=
	\bigO 
	\big( 
	n^{-j}
	\weight_{\mgdnm}(1,1) 
	\big).
\]
\end{lemma}

Figure~\ref{fig:nonstar} displays four multigraphs from~$\zeromgdnm$.
Actually, any multigraph from~$\zeromgdnm$ contains one of those four graphs as a subgraph,
and this property can be described in terms of mandatory edges.
In the following lemma, we use this fact to bound $\marked_{\zeromgdnm}(-1,-1)$.

\begin{lemma} \label{th:large_excess}
Assume~$\DD$ contains at least two integers, has valuation~$r$ and periodicity~$p$.
Let~$m$, $n$ denote two integers tending to infinity,
such that $2m/n$ stays in a fixed compact interval of $]r, \max(\DD)[$,
and~$p$ divides~$2m - n r$, then
\[
    \marked_{\zeromgdnm}(-1,-1)
    =
    \bigO
	\left(
	n^{-1}
	\weight_{\mgdnm}(-1,-1)
	\right).
\]
\end{lemma}

The intuition supporting this proof is that
a multigraph~$G$ belongs to~$\zeromgdnm$ if and only if it contains a vertex~$v$
that is in one of the four configurations depicted in Figure~\ref{fig:nonstar}.
According to Lemma~\ref{th:mandatory_edges},
multigraphs from~$\mgdnm$ that contain those subgraphs
have a negligible total weight.
Now we have all the ingredients to prove Theorem~\ref{th:simple}.

\begin{proof}[Proof of Theorem~\ref{th:simple}]
In Lemma~\ref{th:inclusion_exclusion}, we have proven
that the number of simple graphs in $\sgdnm$ is equal to
$\weight_{\starmgdnm}(-1,-1)$.
By a set manipulation, this quantity can be rewritten
\[
	\weight_{\mgdnm}(-1,-1)
	-
	\weight_{\zeromgdnm}(-1,-1),
\]
where $\zeromgdnm = \mgdnm \setminus \starmgdnm$.
Replacing the second term with the result of Lemma~\ref{th:large_excess}, we obtain
\[
    |\sgdnm| = 
	\weight_{\mgdnm}(-1,-1)
	\big(
	1 + \bigO(n^{-1}) 
	\big).
\]
Finally, the asymptotics of $\weight_{\mgdnm}(-1,-1)$
has been derived in Lemma~\ref{th:asympt_marked}.
\end{proof}

    \section{Random generation}

In order to keep a combinatorial interpretation,
we focused on generating functions~$\set_{\DD}(x)$
with coefficients in~$\{0,1\}$.
Our results hold more generally for any generating function~$D(x)$
with nonnegative coefficients and large enough radius of convergence
(so that the saddle-point from Lemma~\ref{th:large_power} is well defined).
Multigraphs are then counted with a weight that depends
of the degrees of their vertices
\[
    \operatorname{weight}(G) =
    \kappa(G)
    \prod_{v \in V(G)}
    \deg(v)! 
    [x^{\deg(v)}]
    D(x).
\]

The present work has been guided by experiments
on large random graphs with degree constraints.
We used exact and Boltzmann sampling (\cite{DFLS04}).
Observe that to build a random simple graph from~$\sgdnm$,
one can sample multigraphs from~$\mgdnm$
and reject until the multigraph is simple.
As a consequence of Theorem~\ref{th:simple},
the expected number of rejections is~$e^{-W_{\frac{n}{m}}(\zeta)^2 -W_{\frac{n}{m}}(\zeta)}$
(using the notations of the theorem).

    \subsection{Boltzmann sampling} \label{sec:boltzmann}

The construction of the Boltzmann algorithm 
is straightforward from Theorem~\ref{th:multigraphs}.
To build a random multigraph with degrees in~$\DD$,
$n$ vertices and approximately~$m$ edges,
the algorithm first computes a positive value~$x$,
according to the number of edges targeted.
It then draws independently~$n$ integers~$(d_1, \ldots, d_n)$,
following the law
\begin{equation} \label{eq:boltz}
    \mathds{P}(d) = 
    \frac{\left([z^d] D(z)\right) x^d}{D(x)}
\end{equation}
with~$D(x) = \set_{\DD}(x)$.
If their sum is odd, a new sequence is drawn.
Otherwise, the algorithm outputs a random multigraph
with sequence of degrees~$(d_1, \ldots, d_n)$.
To do so, as in the configuration model (\cite{B80}, \cite{Wo78}),
each vertex~$v_i$ receives~$d_i$ half-edges,
and a random pairing on the half-edges is drawn uniformly.

Therefore, the random distribution induced on multigraphs 
by the Boltzmann sampling algorithm
is identical to the configuration model.
Conversely, given a probability distribution on~$\naturals$,
one can choose~$D(x)$ so that the distribution is equal
to the one described by Equation~\eqref{eq:boltz}.
Thus, we expect random multigraphs from the configuration model
and multigraphs with degree constraints
to share many statistical properties.

    \subsection{Recursive method}

For the sampling of a multigraph in~$\mgdnm$, 
the generator first draws a sequence of degrees, 
and then performs a random pairing of half-edges,
as in configuration model and the Boltzmann sampler.
Each sequence~$(d_1, \ldots, d_n)$ from~$\DD^n$ is
drawn with weight
$\prod_{v=1}^n 1/(d_v)!$.
In the first step, we use dynamic programming
to precompute the values~$(S_{i,j})_{0 \leq i \leq n, 0 \leq j \leq 2m}$,
sums of the weights of all the sequences of $i$ degrees that sum to $j$
\[
    S_{i,j} = 
    \sum_{\substack{d_1, \ldots, d_i \in \DD \\ d_1 + \cdots + d_i = j}}
    \prod_{v=1}^i \frac{1}{d_v!},
\]
using the initial conditions and the recursive expression
\[
    S_{i,j} = 
    \begin{cases}
    1 & \text{if $(i,j) = (0,0)$,}\\
    0 & \text{if $i=0$ and $j \neq 0$, or if $j < 0$,}\\
    \sum_{d \in \DD} \frac{S_{i-1,j-d}}{d!}
    & \text{otherwise.}
    \end{cases}
\]
After this precomputation, we generate
the sequence of degrees as follows: first we sample
the last degree~$d_n$ of the sequence
according to the distribution
\[
    \proba(d_n = d) =
    \frac{S_{n-1,2m-d}}{d! S_{n, 2m}},
\]
then we recursively generate the remaining sequence $(d_1, \ldots, d_{n-1})$, 
which must sum to $2m-d_n$. 
Once the sequence of degrees is computed, 
we generate a random pairing on the corresponding half-edges.

    \section{Forthcoming research}

The results presented can be extended in several ways.
The case where~$2m/n$ tends to~$\min(\DD)$ or~$\max(\DD)$ could be considered.
For example, \cite{PW03} have derived, using elementary tools,
the asymptotics of graphs with a lower bound on the minimum degree
when~$m = \bigO(n \log(n))$.
This extension would only require to adjust the saddle-point method from Lemma~\ref{th:large_power}.

We have also derived results on the enumeration of graphs 
where the degree sets vary with the vertices.
The model inputs an infinite sequence of sets~$(\DD_1, \DD_2, \ldots)$
and output graphs where each vertex~$v$ has its degree in~$\DD_v$.
The techniques presented in this paper can be extended to this case,
if some technical conditions are satisfied,
such as the convergence of the series~$n^{-1} \sum_{v \geq 1} \log(\set_{\DD_v}(x))$.
This extension will be part of a longer version of the paper.
Two examples of such families are
graphs with degree parities (\cite{RR82}),
and graphs with a given degree sequence (\cite{BC78}).

We believe that complete asymptotic expansion can be derived
for graphs with degree constraints.
This would require to apply a more general version of Lemma~\ref{th:large_power},
such as presented in Chapter~$4$ by \cite{PW13},
and we would have to consider more complex families than~$\starmgdnm$.

The asymptotics of connected graphs from~$\sgnm$ when~$m-n$ tends to infinity
has first been derived by \cite{BCM90}.
Since then, two new proofs were given,
one by \cite{PW05}, the other by \cite{HS06}.
The proof of Pittel and Wormald relies on a link
between connected graphs and graphs from a particular
family of graphs with degree constraints
(graphs with degrees at least~$2$).
In \cite{ElieThesis}, following the same approach,
but using analytic combinatorics, we obtained
a short proof for the asymptotics of connected multigraphs
from~$\mgnm$ when~$m-n$ tends to infinity.
We now plan to extend this result to simple graphs,
and to derive a complete asymptotic expansion.


In this paper, we have focused on the enumeration of graphs with degree constraints.
We can now start the investigation on the typical structure of random instances of such graphs.
An application would be the enumeration of Eulerian graphs,
\textit{i.e.} connected Euler graphs.

Finally, the inclusion-exclusion technique we used
to remove loops and double edges
can be extended to forbid any family of subgraphs.

\bibliographystyle{abbrvnat}
\bibliography{}

\begin{thebibliography}{18}
\providecommand{\natexlab}[1]{#1}
\providecommand{\url}[1]{\texttt{#1}}
\expandafter\ifx\csname urlstyle\endcsname\relax
  \providecommand{\doi}[1]{doi: #1}\else
  \providecommand{\doi}{doi: \begingroup \urlstyle{rm}\Url}\fi

\bibitem[Bender and Canfield(1978)]{BC78}
E.~A. Bender and E.~Canfield.
\newblock The asymptotic number of labeled graphs with given degree sequences.
\newblock \emph{Journal of Combinatorial Theory, Series A}, 24\penalty0
  (3):\penalty0 296 -- 307, 1978.
\newblock ISSN 0097-3165.

\bibitem[Bender et~al.(1990)Bender, Canfield, and McKay]{BCM90}
E.~A. Bender, E.~R. Canfield, and B.~D. McKay.
\newblock The asymptotic number of labeled connected graphs with a given number
  of vertices and edges.
\newblock \emph{Random Structures and Algorithm}, 1:\penalty0 129--169, 1990.

\bibitem[Bollob\'as(1980)]{B80}
B.~Bollob\'as.
\newblock A probabilistic proof of an asymptotic formula for the number of
  labelled regular graphs.
\newblock \emph{European Journal of Combinatorics}, 1:\penalty0 311--316, 1980.

\bibitem[Chyzak et~al.(2005)Chyzak, Mishna, and Salvy]{CMS05}
F.~Chyzak, M.~Mishna, and B.~Salvy.
\newblock Effective scalar products of d-finite symmetric functions.
\newblock \emph{Journal of Combinatorial Theory, Series A}, 112\penalty0
  (1):\penalty0 1 -- 43, 2005.

\bibitem[de~Panafieu(2014)]{ElieThesis}
E.~de~Panafieu.
\newblock \emph{Analytic Combinatorics of Graphs, Hypergraphs and Inhomogeneous
  Graphs}.
\newblock PhD thesis, Universit\'e Paris-Diderot, Sorbonne Paris-Cit\'e, 2014.

\bibitem[Duchon et~al.(2004)Duchon, Flajolet, Louchard, and Schaeffer]{DFLS04}
P.~Duchon, P.~Flajolet, G.~Louchard, and G.~Schaeffer.
\newblock Boltzmann samplers for the random generation of combinatorial
  structures.
\newblock \emph{Combinatorics, Probability and Computing}, 13:\penalty0 2004,
  2004.

\bibitem[Flajolet and Sedgewick(2009)]{FS09}
P.~Flajolet and R.~Sedgewick.
\newblock \emph{Analytic Combinatorics}.
\newblock Cambridge University Press, 2009.

\bibitem[Flajolet et~al.(1989)Flajolet, Knuth, and Pittel]{FKP89}
P.~Flajolet, D.~E. Knuth, and B.~Pittel.
\newblock The first cycles in an evolving graph.
\newblock \emph{Discrete Mathematics}, 75\penalty0 (1-3):\penalty0 167--215,
  1989.

\bibitem[Janson et~al.(1993)Janson, Knuth, \L{}uczak, and Pittel]{JKLP93}
S.~Janson, D.~E. Knuth, T.~\L{}uczak, and B.~Pittel.
\newblock The birth of the giant component.
\newblock \emph{Random Structures and Algorithms}, 4\penalty0 (3):\penalty0
  233--358, 1993.

\bibitem[Mallows and Sloane(1975)]{MS75}
C.~L. Mallows and N.~J.~A. Sloane.
\newblock Two-graphs, switching classes and euler graphs are equal in number.
\newblock \emph{journal of applied mathematics}, 28\penalty0 (4), 1975.

\bibitem[Pemantle and Wilson(2013)]{PW13}
R.~Pemantle and M.~C. Wilson.
\newblock \emph{Analytic Combinatorics in Several Variables}.
\newblock Cambridge University Press, New York, NY, USA, 2013.

\bibitem[Pittel and Wormald(2003)]{PW03}
B.~Pittel and N.~C. Wormald.
\newblock Asymptotic enumeration of sparse graphs with a minimum degree
  constraint.
\newblock \emph{J. Comb. Theory, Ser. A}, 101\penalty0 (2):\penalty0 249--263,
  2003.

\bibitem[Pittel and Wormald(2005)]{PW05}
B.~Pittel and N.~C. Wormald.
\newblock Counting connected graphs inside-out.
\newblock \emph{Journal of Combinatorial Theory, Series B}, 93\penalty0
  (2):\penalty0 127--172, 2005.

\bibitem[Read and Robinson(1982)]{RR82}
R.~Read and R.~Robinson.
\newblock Enumeration of labelled multigraphs by degree parities.
\newblock \emph{Discrete Mathematics}, 42\penalty0 (1):\penalty0 99 -- 105,
  1982.

\bibitem[Robinson(1969)]{RWR69}
R.~W. Robinson.
\newblock Enumeration of euler graphs.
\newblock \emph{Proof Techniques in Graph Theory}, pages 147--153, 1969.

\bibitem[van~der Hofstad(2014)]{Ho14}
R.~van~der Hofstad.
\newblock \emph{Random Graphs and Complex Networks. Vol. I}.
\newblock 2014.

\bibitem[van~der Hofstad and Spencer(2006)]{HS06}
R.~van~der Hofstad and J.~Spencer.
\newblock Counting connected graphs asymptotically.
\newblock \emph{European Journal on Combinatorics}, 26\penalty0 (8):\penalty0
  1294--1320, 2006.

\bibitem[Wormald(1978)]{Wo78}
N.~Wormald.
\newblock \emph{Some problems in the enumeration of labelled graphs}.
\newblock Newcastle University, 1978.

\end{thebibliography}
\appendix
\section{Proofs}

In this appendix, we include the proofs of the lemmas and theorems.

\begin{proof}[Proof of Theorem~\ref{th:multigraphs}]
By definition of the compensation factor, the number of multigraphs of the theorem is equal to
\[
    \frac{1}{2^m m!}
    \sum_{G \in \mgdnm}
    \orderings(G).
\]
Let us consider an ordering
\[
    ((v_1, v_2), (v_2, v_3), \ldots, (v_{2m-1}, v_{2m})).
\]
of a multigraph~$G$ from $\mgdnm$.
For all~$1 \leq i \leq n$, let $P_i = \{ j\ |\ v_j = i\}$ 
denote the set of positions of the vertex~$i$ in this ordering.
Since the vertices have their degrees in~$\DD$, each $P_i$ has size in~$\DD$.
This implies a bijection between
\begin{itemize}
\item
the orderings of multigraphs in $\mgdnm$, 
\item
the sequences of sets $(P_1, \ldots, P_n)$,
where the size of each set is in~$\DD$,
and $(P_1, \ldots, P_n)$ is a partition of~$\{1, \ldots, 2m\}$
(\textit{i.e.} the sets are disjoints 
and $\bigcup_{i=1}^n P_i = \{1, \ldots, 2m\}$).
\end{itemize}
We now interpret $(P_1, \ldots, P_n)$ as a sequence of sets that contain labelled objects
and apply the \emph{Symbolic Method} (see \cite{FS09}).
The exponential generating function of sets of size in~$\DD$ is $\set_{\DD}(x)$.
The bijection then implies
\[
    \sum_{G \in \mgdnm}
    \orderings(G)
    =
    (2m)! [x^{2m}] \set_{\DD}(x)^n,
\]
and the theorem follows, after division by~$2^m m!$.
\end{proof}

\begin{proof}[Proof of Lemma~\ref{th:inclusion_exclusion}]
As explained in the paragraphs \textbf{First markink} and \textbf{Second marking}
of Section~\ref{sec:inclusion_exclusion_star}, the following relations hold
\begin{align*}
    \marked^{(1)}_{\starmgdnm}(0,0) &= |\sgdnm|,\\
    \marked^{(2)}_{\starmgdnm}(u,v) &= \marked^{(1)}_{\starmgdnm}(u+1,v+1).
\end{align*}
Furthermore, by construction of~$\starmgdnm$, we have
\[
    \marked_{\starmgdnm}(u,v)
    =
    \marked_{\starmgdnm}^{(2)}(u,v),
\]
so~$\marked_{\starmgdnm}(-1,-1) = |\sgdnm|$.
\end{proof}

\begin{proof}[Proof of Lemma~\ref{th:exact_marked}]
To build an ordering of a multigraph from $\marked_{\mgdnm}$ 
with $2 k$ vertices in marked double edges
and $\ell$ vertices in marked loops,
we perform the following steps:
\begin{enumerate}
\item
choose the labels of the $2 k$ vertices that appear in the marked double edges,
and the $\ell$ vertices that appear in the marked loops.
There are $\binom{n}{2 k, \ell, n - 2k - \ell}$ such choices.
\item
choose the distinct $k$ edges of distinct vertices, 
among the chosen~$2k$ vertices,
that will become the marked double edges.
There are $\frac{(2 k)!}{2^k k!}$ such choices.
\item
order the $2 k$ marked double edges
and the vertices in each of them.
There are $\frac{(2k)! 4^k}{2^k}$ ways to order them.
\item
order the $\ell$ loops.
There are $\ell!$ ways to do so.
\item
choose among the $m$ edges of the final ordering
which ones receive marked loops and which ones receive marked double edges.
There are $\binom{m}{2k, \ell, m-2k-\ell}$ choices.
\item
to fill the rest of the final ordering,
build an ordering of length $2m - 4 k - 2 \ell$
where the $2 k$ vertices that belong to marked double edges and
the $\ell$ vertices that appear in marked loops have degree in~$\DD - 2$,
while the other $n - 2k - \ell$ vertices have degree in~$\DD$.
The number of such orderings is 
$(2m - 4k - 2\ell)!
    [x^{2m - 4k - 2\ell}]
    \set_{\DD-2}(x)^{2k+\ell}
    \set_{\DD}(x)^{n - 2k - \ell}$.
\end{enumerate}
This bijective construction implies the following enumerative result
\begin{align*}
    \sum_{G \in \marked(\mgdnm)}
    &
    \kappa(G)
    u^{k(G)}
    v^{\ell(G)}\\
    =
    \frac{1}{2^m m!}
    \sum_{k,\ell \geq 0}
    &
    \binom{n}{2 k, \ell, n - 2k - \ell}
    \frac{(2 k)!}{2^k k!}
    \frac{(2k)! 4^k}{2^k}
    \ell!
    \binom{m}{2k,\ell, m-2k-\ell}
    \\
    &
    (2m - 4k - 2\ell)!
    [x^{2m - 4k - 2\ell}]
    \set_{\DD-2}(x)^{2k+\ell}
    \set_{\DD}(x)^{n - 2k - \ell}
    u^{k(G)}
    v^{\ell(G)}.    
\end{align*}
After simplification, this last expression can be rewritten
\[
	\weight_{\mgdnm}(u,v)
	=
    \frac{(2m)!}{2^m m!}
    [x^{2m}]
    \Bigg(
    \sum_{k,\ell \geq 0}
    a_{n,m,2k+\ell}
    \frac{(u W_{\frac{n}{m}}(x)^2)^k}{k!} 
    \frac{(v W_{\frac{n}{m}}(x))^{\ell}}{\ell!}
    \Bigg)
    \set_{\DD}(x)^n.
\]
\end{proof}

\begin{proof}[Proof of Lemma~\ref{th:technical}]
Developing the exponential as a double sum
\[
    e^{u+v}
    =
    \sum_{k, \ell \geq 0}
    \frac{u^k}{k!}
    \frac{v^{\ell}}{\ell!},
\]
the result can be rewritten
\[
	n \frac{| 1 - a_{n,m,2k+\ell} |}{k! \ell!}
	\leq
	[u^k v^\ell]
	C(u,v)
\]
for all~$k$, $\ell$.
We prove that when~$n$ is large enough,
we have
\begin{equation} \label{eq:bounded_coef}
	n \frac{| 1 - a_{n,m,2k+\ell} |}{k! \ell!}
	\leq
	\left( 1 + \frac{n}{m} \right)
	\frac{(2k+\ell)^2 e^{4k + 2\ell}}{\sqrt{k! \ell!}}
\end{equation}
for all~$k, \ell \geq 1$.
Since the right-hand side are the coefficients
of a function analytic on $\complex^2$, this will conclude the proof.

Let~$b_{n,j}$ denote the value
$
    \prod_{i=0}^{j-1} \left( 1 - \frac{i}{n} \right),
$
then observe that $a_{n,m,j}$ is equal to 
$
	b_{n,j} b_{m,j}/b_{2m, 2j}.
$
Since $b_{n,j} \leq 1$, if $(c_{n,j})$ denotes a sequence such that
$c_{n,j} \leq b_{n,j}$ for all $(n,j)$, then
$
    c_{n,j} c_{m,j}
    \leq
    a_{n,m,j}
    \leq
    c_{2m,2j}^{-1}
$,
which implies
\begin{equation} \label{eq:technical_ineq}
    n \frac{|1-a_{n,m,2k+\ell}|}{k! \ell!}
    \leq
    n
    \frac{\max(c_{2m,4k+2\ell}^{-1}-1, 1-c_{n,2k+\ell} c_{m,2k+\ell})}{k! \ell!}.
\end{equation}
We now prove that Equation~\eqref{eq:bounded_coef} holds
both for $2k+\ell \leq \sqrt{m}/2$ and for $2k+\ell > \sqrt{m}/2$.

\textbf{Case $2k+\ell \leq \sqrt{m}/2$.}
We prove by recurrence on~$j$ that
$
    b_{n,j} \geq 1 - \frac{j^2}{n}.
$
The recurrence is initialized with $b_{n,0} = 1$.
Assuming it is satisfied at rank~$j$, then
\[
    b_{n,j+1}
    = \left( 1 - \frac{j}{n} \right) b_{n,j}
    \geq
    \left( 1 - \frac{j}{n} \right) \left( 1 - \frac{j^2}{n} \right)
    \geq
    1 - \frac{(j+1)^2}{n},
\]
which concludes the proof of the recurrence.
This implies, using Inequality~\eqref{eq:technical_ineq},
\[
    n \frac{|1-a_{n,m,2k+\ell}|}{k! \ell!} 
	\leq
    \frac{n}{k! \ell!}
    \max \left(
    \frac{1}{1 - \frac{(4k+2\ell)^2}{2m}} - 1,
    1 - \left(1 - \frac{(2k+\ell)^2}{n} \right) \left(1 - \frac{(2k+\ell)^2}{m} \right)
    \right).
\]
Since $2k+\ell \leq \sqrt{m}/2$, the first argument of the maximum function
is at most~$1$. The second argument is smaller than~$(n^{-1} + m^{-1})(2k+\ell)^2$.
Therefore, we have
\[
    n \frac{|1-a_{n,m,2k+\ell}|}{k! \ell!} 
	\leq
    \left( 1 + \frac{n}{m} \right)
    \frac{(2k+\ell)^2}{k! \ell!},
\]
and Inequality~\eqref{eq:bounded_coef} is satisfied.

\textbf{Case $2k+\ell > \sqrt{m}/2$.}
We first prove $b_{n,j} \geq e^{-j}$.
To do so, we apply a sum-integral comparison in the expression
\[
    \log(b_{n,j}) 
    = \sum_{i=0}^{j-1} \log \left( 1-\frac{i}{n} \right)
    \geq
    \int_{0}^j
    \log \left( 1-\frac{x}{n} \right)
    d x
    = - (n-j) \log \left( 1 - \frac{j}{n} \right) -j
    \geq 
    -j.
\]
Inequality~\eqref{eq:technical_ineq} then implies
\[
    n \frac{|1-a_{n,m,2k+\ell}|}{k! \ell!}
    \leq
    \frac{n}{k! \ell!} 
    \max \left(
        e^{4k+2\ell}-1, 1-e^{-(4k+2\ell)}
    \right)
    \leq
    \frac{n}{\sqrt{k! \ell!}}
	\frac{e^{4k+2\ell}}{\sqrt{k! \ell!}}.
\]
We now prove that
$n / \sqrt{k! \ell!}$ is smaller than~$1$ for~$n$ large enough.
Indeed, $2k+\ell > \sqrt{m}/2$ implies $\max(k,\ell) \geq \sqrt{m}/8$, so
\[
    \frac{n}{\sqrt{k! \ell!}}
    \leq
    \frac{n}{\sqrt{\max(k,\ell)!}}
    \leq
    \frac{n}{(\sqrt{m}/8)!},    
\]
and since~$m/n$ stays in a compact interval of~$\reals_{>0}$,
this last term tends to~$0$ with~$n$.
We then conclude
\[
    n \frac{|1-a_{n,m,2k+\ell}|}{k! \ell!}
    \leq
    \frac{e^{4k+2\ell}}{\sqrt{k! \ell!}}
\]
for~$n$ large enough, so Inequality~\eqref{eq:bounded_coef} is satisfied.
\end{proof}

\begin{proof}[Proof of Lemma~\ref{th:asympt_marked}]
We start with the expression of $\weight_{\mgdnm}(u,v)$ derived in Lemma~\ref{th:exact_marked}
\[
    \weight_{\mgdnm}(u,v) 
	=
    \frac{(2m)!}{2^m m!}
    [x^{2m}]
    \Bigg(
    \sum_{k,\ell \geq 0}
    a_{n,m,2k+\ell}
    \frac{(u W_{\frac{n}{m}}(x)^2)^k}{k!}
    \frac{(v W_{\frac{n}{m}}(x))^{\ell}}{\ell!}
    \Bigg)
    \set_{\DD}(x)^n.
\]
Using the notation
\[
	A(x) = 
	e^{u W_{\frac{n}{m}}(x)^2 + v W_{\frac{n}{m}}(x)}
	-
	\sum_{k,\ell \geq 0}
	a_{n,m,2k+\ell}
	\frac{(u W_{\frac{n}{m}}(x)^2)^k}{k!}
	\frac{(v W_{\frac{n}{m}}(x))^{\ell}}{\ell!},
\]
this implies
\[
	\frac{(2m)!}{2^m m!}
    [x^{2m}]
	e^{u W_{\frac{n}{m}}(x)^2 + v W_{\frac{n}{m}}(x)}
	\set_{\DD}(x)^n
	-
	\weight_{\mgdnm}(u,v)
	=
	\frac{(2m)!}{2^m m!}
    [x^{2m}]
	A(x)
	\set_{\DD}(x)^n.
\]
Observe that~$W_{\frac{n}{m}}(x)$ has valuation~$0$ and period~$p$.
According to Lemma~\ref{th:large_power}, we have
\[
	\frac{(2m)!}{2^m m!}
    [x^{2m}]
	e^{u W_{\frac{n}{m}}(x)^2 + v W_{\frac{n}{m}}(x)}
	\set_{\DD}(x)^n
	=
	\frac{(2m)!}{2^m m!}
	\frac{\set_{\DD}(\zeta)^n}{\zeta^{2m}}
	\frac{p}{\sqrt{2 \pi n \zeta \phi'(\zeta)}}
    e^{u W_{\frac{n}{m}}(\zeta)^2 + v W_{\frac{n}{m}}(\zeta)}
	(1 + \bigO(n^{-1})),
\]
so the demonstration is complete if we prove
\[
	\frac{(2m)!}{2^m m!}
    [x^{2m}]
	A(x)
	\set_{\DD}(x)^n
	=
	\frac{(2m)!}{2^m m!}
	\bigO \left(
	n^{-1}
	\frac{\set_{\DD}(\zeta)^n}{\zeta^{2m} \sqrt{n}}
	\right).
\]
The Taylor coefficients of $W_{\frac{n}{m}}(x)$ need not be positive,
so we introduce the entire function
\[
    \WW_{\frac{n}{m}}(x) = \sum_{n \geq 0} | [z^n] W_{\frac{n}{m}}(z) | x^n,
\]
which dominate $W_{\frac{n}{m}}(x)$ coefficient-wise.
By application of Lemma~\ref{th:technical}, $\frac{1}{n} C(u \WW_{\frac{n}{m}}(x)^2, v \WW_{\frac{n}{m}}(x))$
dominates coefficient-wise~$A(x)$,
and therefore
\[
	\left|
	\frac{(2m)!}{2^m m!}
    [x^{2m}]
	A(x)
	\set_{\DD}(x)^n
	\right|
	\leq
	\frac{(2m)!}{2^m m!}
    [x^{2m}]
	\frac{1}{n} C(u \WW_{\frac{n}{m}}(x)^2, v \WW_{\frac{n}{m}}(x))
	\set_{\DD}(x)^n.
\]
Finally, according to Lemma~\ref{th:large_power}, we have
\[
	\frac{(2m)!}{2^m m!}
    [x^{2m}]
	\frac{1}{n} C(u \WW_{\frac{n}{m}}(x)^2, v \WW_{\frac{n}{m}}(x))
	\set_{\DD}(x)^n
	=
	\frac{(2m)!}{2^m m!}
	\bigO \left(
	n^{-1} 
	\frac{\set_{\DD}(\zeta)^n}{\zeta^{2m} \sqrt{n}}
	\right).
\]
\end{proof}

\begin{proof}[Proof of Lemma~\ref{th:mandatory_edges}]
Let $\tilde{R}$ denote the set of multigraphs from $\mgdnm$
with~$j$ distinguished mandatory edges
\[
	e_1 = \{v_1, v_2\}, \ldots, e_j = \{v_{2j-1},v_{2j}\}.
\]
Given an ordering of a multigraph from~$R$,
we can distinguish the first occurrences of the mandatory edges,
in order to obtain the ordering of a multigraph from~$\tilde{R}$.
Therefore, the number of orderings of multigraphs from~$R$
is at most equal to the number of orderings of multigraphs from~$\tilde{R}$.
Dividing by~$2^m m!$, this implies
\[
    \sum_{G \in R} \kappa(G)
    \leq
    \sum_{G \in \tilde{R}} \kappa(G),
\]
so~$\marked_{R}(1,1) \leq \marked_{\tilde{R}}(1,1)$.

Let~$W$ denote the fixed set of vertices that appear in the mandatory edges,
and for all~$w \in W$, let $d_w$ denote the number of occurrences of the vertex~$w$
in the mandatory edges
\[
	d_w = \big| \{i\ | v_i = w\} \big|.
\]
Let also $G^{(d)}_{n,m}$ denote the set of multigraphs
with~$n$ vertices and $m$ edges,
where each vertex~$w$ from the mandatory edges has degree in $\DD - d_w$
and the other vertices have degrees in~$\DD$.
To construct an ordering from a multigraph in~$\marked_{\tilde{R}}$,
we choose the~$j$ positions of the mandatory edges among the~$m$ positions available, 
the order of the vertices in those edges,
and mark or not each of them.
Then the rest of the ordering is filled
with an ordering from $\marked_{G^{(d)}_{n,m-j}}$.
Therefore, the number of orderings from $\marked_{\tilde{R}}$ is at most
\[
	m^j 2^j 2^j
	2^{m-j} (m-j)!
	\marked_{G^{(d)}_{n,m-j}}(1,1).
\]
Dividing by $2^m m!$ and using the fact that~$j$ is fixed, we obtain
\begin{equation} \label{eq:tildeR}
	\weight_{\tilde{R}}(1,1)
	=
	\bigO \left( \marked_{G^{(d)}_{n,m-j}}(1,1) \right).
\end{equation}
Following the steps of Lemma~\ref{th:exact_marked},
$\marked_{G^{(d)}_{n,m-j}}(1,1)$ is smaller than or equal to
\[
	\frac{(2m-2j)!}{2^{m-j} (m-j)!}
	[x^{2m-2j}]
	\Bigg(
	\sum_{k,\ell \geq 0}
	a_{n,m-j,2k+\ell}
	\frac{W_{\frac{n}{m}}(x)^{2k+\ell}}{k!\ell!}
	\Bigg)
	\Bigg(
	\prod_{v \in W}
	\set_{\DD-d_v}(x)
	\Bigg)
	\set_{\DD}(x)^{n-|W|}.
\]
An application of the same argument as in the proof of Lemma~\ref{th:asympt_marked} leads to
\[
	\weight_{G^{(d)}_{\DD}(n, m - j)}(1,1)
	=
	\frac{(2m-2j)!}{2^{m-j} (m-j)!}
	\bigO
	\left(
	\frac{\set_{\DD}(\zeta)^{n-|W|}}{\zeta^{2(m-j)} \sqrt{n-|W|}}
	\right).
\]
Since~$|W|$ and~$j$ are fixed, this implies, using Lemma~\ref{th:asympt_marked},
\[
	\marked_{G^{(d)}_{n,m-j}}(1,1)
	=
	\frac{(2m-2j)!}{2^{m-j} (m-j)!}
	\frac{2^m m!}{(2m)!}
	\bigO
	\left(
	\marked_{\mgdnm}(1,1)
	\right).
\]
Simplifying and injecting this relation from Equation~\eqref{eq:tildeR}, we obtain
\[
	\marked_{\tilde{R}}(1,1)
	=
	\bigO
	\big(
	n^{-j}
	\marked_{\mgdnm}(1,1)
	\big).
\]
\end{proof}

\begin{proof}[Proof of Lemma~\ref{th:large_excess}]
By definition, a multigraph~$G$ belongs to~$\zeromgdnm$ if and only if it contains a vertex~$v$
that is in one of the following configurations:
\begin{enumerate}
\item
the loop \{v,v\} appears at least twice in $E(G)$,
\item
there is a vertex~$u$ such that the edge \{u,v\} appears at least three times,
\item
there is a vertex~$u$ such that \{v,v\} is in~$E(G)$ and \{u,v\} appears at least twice,
\item
there are vertices~$u$ and~$w$ such that \{u,v\} and \{v,w\} both appear at least twice.
\end{enumerate}
Let~$\tilde{R}_1$ (resp. $\tilde{R}_2$, $\tilde{R}_3$, $\tilde{R}_4$)
denote the set of multigraphs from~$\mgdnm$ that contain a vertex in configuration~$1$
(resp.~$2$, $3$, $4$). We then have
\[
    \zeromgdnm = \tilde{R}_1 \cup \tilde{R}_2 \cup \tilde{R}_3 \cup \tilde{R}_4.
\]
Let also $R_1$, $R_2$, $R_3$ and $R_4$ denote four subsets of $\mgdnm$, such that
\begin{enumerate}
\item
the multigraphs from $R_1$ contain
two occurrences of the loop $\{1,1\}$,
\item
the multigraphs from $R_2$ contain
three occurrences of the edge $\{1,2\}$,
\item
the multigraphs from $R_3$ contain
an occurrence of $\{1,1\}$
and two occurrences of $\{1,2\}$,
\item
the multigraphs from $R_4$ contain
two occurrences of $\{1,2\}$
and two occurrences of $\{1,3\}$
\end{enumerate}
(see Figure~\ref{fig:nonstar}).
Given the symmetric roles of the vertices,
the number of orderings from multigraphs in~$\tilde{R}_1$
(resp. $\tilde{R}_2$, $\tilde{R}_3$, $\tilde{R}_4$)
is lesser than or equal to~$n$ times (resp. $n^2$, $n^2$, $n^3$)
the number of orderings from multigraphs in~$R_1$ (resp. $R_2$, $R_3$, $R_4$).
This implies
\begin{align*}
    \marked_{\tilde{R}_1}(1,1) &\leq n \marked_{R_1}(1,1),\\
    \marked_{\tilde{R}_2}(1,1) &\leq n^2 \marked_{R_2}(1,1),\\
    \marked_{\tilde{R}_3}(1,1) &\leq n^2 \marked_{R_3}(1,1),\\
    \marked_{\tilde{R}_4}(1,1) &\leq n^3 \marked_{R_4}(1,1),
\end{align*}
so
\[
	\marked_{\zeromgdnm}(1,1)
	\leq
	n \marked_{R_1}(1,1)
	+
	n^2 \marked_{R_2}(1,1)
	+
	n^2 \marked_{R_3}(1,1)
	+
	n^3 \marked_{R_4}(1,1).
\]
The multigraphs from~$R_1$ (resp. $R_2$, $R_3$, $R_4$)
have $2$ mandatory edges (resp. $3$, $3$, $4$).
Four applications of Lemma~\ref{th:mandatory_edges} lead to
\[
	\marked_{\zeromgdnm}(1,1)
	= 
	\bigO(n^{-1})
	\weight_{\mgdnm}(1,1).
\]
Finally, according to Lemma~\ref{th:asympt_marked},
\[
	\weight_{\mgdnm}(1,1)
	=
	\bigO \left( \weight_{\mgdnm}(-1,-1) \right).
\]
\end{proof}

\end{document}